\documentclass{article}[11]
\usepackage{amsmath, amssymb, color}

\newtheorem{thm}{Theorem}[section]
\newtheorem{prop}[thm]{Proposition}
\newtheorem{lem}[thm]{Lemma}
\newtheorem{cor}[thm]{Corollary}
\newtheorem{definition}[thm]{Definition}

\newtheorem{remark}[thm]{Remark}
\newenvironment{proof}[1][,]{\medskip\ifcat,#1
\noindent{{\it Proof}:\ \ }\else\noindent{\it Proof of #1.\ }\fi}
{\hfill$\square$\medskip}

\DeclareMathOperator\vol{vol}

\DeclareMathOperator\GL{GL}
\DeclareMathOperator\gl{gl}

\DeclareMathOperator\Ree{Re}
\DeclareMathOperator\Imm{Im}

\DeclareMathOperator\Len{\mathit{l}}

\DeclareMathOperator\D{\mathcal{D}}

\def\C{\mathbb{C}}

\def\Sph{\mathbb{S}}
\def\S{\mathcal{S}}
\def\R{\mathbb{R}}
\def\Z{\mathbb{Z}}

\begin{document}

\title{Extremal length in higher dimensions and complex systolic inequalities}

\author{Tommaso Pacini\footnote{University of Torino, Italy,
\texttt{tommaso.pacini@unito.it}}}

%\date

\maketitle

\begin{abstract} Extremal length is a classical tool in 1-dimensional complex analysis for building conformal invariants. We propose a higher-dimensional generalization for complex manifolds and provide some ideas on how to estimate and calculate it. We also show how to formulate certain natural geometric inequalities concerning moduli spaces in terms of a complex analogue of the classical Riemannian notion of systole. 
\end{abstract}

\section{Introduction}\label{s:intro}

Let $\D\subseteq\C$ be an open domain. In the 1940s, Ahlfors and Beurling (building upon previous ideas of Gr\"otzsch) proposed the following method for constructing conformal invariants. Choose a set of rectifiable curves $\Gamma$ in $\mathcal{D}$ (with boundary in $\overline\D$). Given any positive function $\rho:\D\rightarrow\R$, \textit{i.e.} conformal metric $\rho^2\cdot g_{std}$ on $\D$, set 
\begin{equation*}
\Len(\Gamma,\rho):=\inf_{\gamma\in \Gamma} \int_\gamma \rho\,|dz|, \ \ A(\rho):=\int_{\D} \rho^2 dxdy
\end{equation*}
then define the \textit{extremal length} of $\Gamma$ as
\begin{equation}\label{eq:extremallength}
 \mu_\Gamma:=\sup_{\rho}\frac{(\Len(\Gamma,\rho))^2}{A(\rho)}.
\end{equation}
Here, we restrict our attention to those ``admissible'' $\rho$ which induce a finite, positive area $A(\rho)$.

The result is a number in the interval $[0,\infty]$, but the possibility of extreme values is reduced by the choice of quotient $\Len^2/A$, concocted to be scale invariant: given a constant $c>0$, it does not distinguish between $\rho$ and $c\rho$. Since the set of conformal metrics is preserved under biholomorphisms $\phi:\D\rightarrow \D'$, one automatically obtains $\mu_\Gamma=\mu_{\Gamma'}$ when $\Gamma'=\phi(\Gamma)$. More generally, if we allow $\rho$ to have zeros we obtain $\mu_\Gamma\leq \mu_{\Gamma'}$ for any $\phi$ holomorphic.

These invariants allow one to classify and distinguish not only domains, \textit{e.g.} annuli, but also domains plus certain configurations of internal or boundary points, \textit{e.g.} quadrilaterals, incorporated through judicious choices of $\Gamma$. We refer to \cite{Ahlfors} for details and examples. In this regard we emphasize that, given the dependence on the class $\Gamma$, the natural context in which to apply this invariant is when it is a priori possible to understand how such classes change under the action of diffeomorphisms or maps. This might be achieved by choosing $\Gamma$ to be a homology class, or by tracking the configuration of boundary points that define $\Gamma$. Another application arises when studying holomorphic structures, already known to be different, on the same domain. In this case we can fix $\Gamma$ once and for all; the invariant then furnishes a parameter on the moduli space.

More generally, across the years, the concept of extremal length has found a wide variety of applications. We again refer to \cite{Ahlfors} for further details.

\paragraph{Extremal volume.}It is an interesting question to find an analogous ``extremal volume'' in higher dimensions. One quickly realizes that addressing this question requires imposing, a priori, a strong subjective viewpoint on the whole theory. In dimension 1, complex analysis is intimately intertwined with conformal geometry. Extremal length relies on this ambiguity by measuring lengths via metrics which are introduced using conformal factors governed by the underlying complex structure. In higher dimensions there is no such relationship. One must thus make a choice whether to prefer the complex-analytic or metric viewpoint, each at the expense of the other. Our proposal is based on the following \textit{Ans\"atze}:
\begin{itemize}
 \item The new construction should reduce to the classical one for domains in $\C$.
 \item In $n$ complex dimensions, curves should be replaced by submanifolds of real dimension $n$.
 \item The new theory should be of a purely complex-analytic nature. 
\end{itemize}
The first two conditions are (arguably!) uncontroversial. We achieve the third by replacing conformal metrics with complex volume forms. Concerning this point, an obvious metric-oriented alternative might be to work in terms of K\"ahler metrics, thus adopting a strongly Riemannian, or perhaps symplectic, viewpoint. We remark that some developments of extremal length, such as the theory of quasi-conformal mappings, have been extended in higher dimensions by adopting a metric point of view and completely dropping the complex structure, cf. \cite{HeinonenKoskela}. One of our aims, however, cf. Sections \ref{s:extremalvolume} and \ref{s:estimates}, is to show that the geometry of real vs. complex volume forms is sufficiently rich to generate an interesting, purely complex-analytic theory, even without the use of metrics.

The result of our construction is a holomorphic invariant which depends only the submanifold geometry of the ambient space, cf. Definition \ref{def:extremalvol}.

\paragraph{Further features.}We wish to emphasize two more aspects of this construction.
\begin{itemize}
\item Roughly speaking, our invariant depends only on the space of ``totally real" submanifolds, paying no attention to the more usual complex submanifolds. It seems to us that totally real geometry encodes the complex structure in a different way, and it seems worthwhile to further develop its role within complex analysis. 
\item An alternative way of describing our construction, which better underlines the interplay between differential geometry and complex analysis, is the following. It is the result of looking at the well-known geometry of Calabi-Yau manifolds and special Lagrangian submanifolds \cite{HL} and stripping away all metric and symplectic information, so as to expose its purely holomorphic backbone. We then call upon the idea of extremal length to repackage these ingredients in the form of invariants. The resulting construction applies to any complex manifold.
\end{itemize}

\paragraph{Calculations.}Of course, it is important that these invariants be calculable. Recall the situation in dimension 1: the modulus is defined for any domain $\mathcal{D}\subseteq\C$ and any $\Gamma$, but in general one can only hope to approximate or bound its value. In order to calculate it precisely, it is usually necessary to first apply the Riemann mapping theorem, bringing $\mathcal{D}$ into some ``standard form'', then use special properties of this standard form to perform the calculations.

In higher dimensions there is no analogue of the Riemann mapping theorem. Whatever the holomorphic invariant, the best one can thus probably hope for is to calculate it in the case of manifolds with some special structure. In Section \ref{s:extremal} we describe a model situation in which this is possible for our invariant, cf. Theorem \ref{thm:calculate}. In Section \ref{s:first_examples} we test this result on Reinhardt domains and elliptic fibre bundles, also providing some comparisons with Calabi-Yau and special Lagrangian geometry. In Section \ref{s:example} we calculate the extremal volume of complex tori with respect to any class of submanifolds defined by a homology class.

\paragraph{Complex systolic inequalities.}Working with tori and their moduli spaces suggests the following development. One of the simplest applications of the concept of extremal length concerns geometric inequalities: the classical example is a theorem of Loewner from 1949, concerning the relationship between geodesics and area on Riemannian tori, which has a quick proof in terms of extremal length. A more modern formulation of such inequalities is in terms of systolic geometry. In complete analogy, in Section \ref{s:cpx_systole} we show that our concept of extremal volume triggers a notion of complex systoles, cf. Definition \ref{def:complexsystole}. Theorem \ref{thm:tori_systolicineq} shows how a certain geometric feature of the moduli space of polarized complex tori can be expressed in terms of a complex systolic inequality.

\paragraph{Acknowledgements.} The question of how to define a higher-dimensional analogue of extremal length was mentioned to me by Eric Bedford. The notion presented here rests upon previous work on totally real submanifolds joint with Jason Lotay. While thinking about geometric inequalities related to extremal volume and Theorem \ref{thm:tori_systolicineq}, I came across \cite{fan}, which defines a more restrictive notion of systoles for Calabi-Yau manifolds,  and \cite{haiden}, which proves essentially the same result as Theorem \ref{thm:tori_systolicineq}, but with no relation to extremal length and with a focus on symplectic rather than complex geometry. I thus realized that systolic geometry provides a natural context for such inequalities. Finally, I wish to thank the organizers and participants of the conference ``Complex Analysis and Geometry - XXIV'' in Levico Terme, Italy, where I first presented these ideas, for interesting conversations.

\section{Extremal volume}\label{s:extremalvolume}

Let $(M,J)$ be a complex manifold of complex dimension $n$. Let $K_M$ denote the holomorphic line bundle of differential forms of type $(n,0)$ and let $\Omega$ be any smooth section of $K_M$. We then obtain the following data:

\begin{itemize}
\item Using complex conjugation we can construct the real $2n$-form
\begin{equation*}
\Omega_M:=(-1)^{\frac{n(n-1)}{2}}\left(\frac{i}{2}\right)^n\,\Omega\wedge\bar\Omega\in\Lambda^{2n}(M;\R). 
\end{equation*}
In local holomorphic coordinates, $\Omega=f(z,\bar z)\,dz^1\wedge\dots\wedge dz^n$ and $\Omega_M=|f|^2 dx^1\wedge dy^1\wedge\dots\wedge dx^n\wedge dy^n$. If $\Omega$ is nowhere vanishing, $K_M$ is differentiably trivial and $\Omega_M$ is a real volume form on $M$, compatible with the standard orientation on $M$ induced by $J$.
\item Let $\pi\leq T_pM$ be an oriented plane of real dimension $n$. To define a $n$-form on $\pi$ it suffices to define its value on a positive basis $v_1,\dots, v_n$: the rest follows from multi-linearity. Taking the norm of the value of $\Omega$ we thus define a real $n$-form on $\pi$ as follows:
\begin{equation*}
 \Omega_\pi(v_1,\dots,v_n):=|\Omega[p](v_1,\dots,v_n)|.
\end{equation*}
It vanishes in two cases: either when $\Omega[p]=0$, or when $v_1,\dots,v_n$ are not $\C$-linearly independent, \textit{i.e.} $\pi$ contains a complex line.
\end{itemize}

Given an oriented submanifold $L\subseteq M$ (possibly with boundary) of real dimension $n$, we obtain a $n$-form $\Omega_L$ on $L$ by setting $\Omega_L[p]:=\Omega_\pi$ where $\pi=T_pL$, for any $p\in L$. In general it is a $C^0$-section of $\Lambda^n(L;\R)$. 

\begin{definition} We call $A(\Omega):=\int_M\Omega_M\geq 0$ the $\Omega$-volume of $M$. We call $\int_L\Omega_L\geq 0$ the $\Omega$-volume of $L$, and $L\mapsto\int_L\Omega_L$ the $\Omega$-\textit{volume} functional. 

Given a set $\Lambda$ of oriented submanifolds $L\subseteq M$ (possibly with boundary) of real dimension $n$, we let $\Len(\Lambda,\Omega):=\inf_{L\in \Lambda} \int_L\Omega_L$ denote the infimum value of the $\Omega$-volume functional restricted to $\Lambda$.
\end{definition}

Notice that any function $e^{i\theta}:M\rightarrow \Sph^1$ defines a ``rotated" complex form $\Omega':=e^{i\theta}\Omega$: we will say that $\Omega$, $\Omega'$ are \textit{equivalent}. The above constructions do not detect the difference between these forms: $\Omega_M=\Omega_M'$ and $\Omega_\pi=\Omega_\pi'$. In particular, the $\Omega$-volume depends only on the equivalence class of $\Omega$. 

We can now present our concept of extremal volume. The quantity $\Len^2/A$, defined above, is an invariant of the triple $(M,\Omega,\Lambda)$. It is also invariant under rescalings and rotations of $\Omega$. To obtain an invariant depending only on $(M,\Lambda)$ we adopt the strategy used for extremal length.

\begin{definition}\label{def:extremalvol}
We define the \textit{extremal volume} of $\Lambda$ as
\begin{equation*}
 \mu_\Lambda:=\sup_{\Omega}\frac{(\Len(\Lambda,\Omega))^2}{A(\Omega)},
\end{equation*}
where we restrict our attention to those ``admissible" $\Omega$ such that $0<A(\Omega)<\infty$.
\end{definition}

Given any biholomorphism $\phi:M\rightarrow M'$, choosing $\Lambda'=\{\phi(L):L\in\Lambda\}$ it is clear that $\mu_\Lambda=\mu_{\Lambda'}$.

We must check that the new invariant coincides with the classical one in the case where $M=\D$ is a domain in $\C$. In this case any $\Omega$ may be written as $\Omega=f\,dz$ for some $f:\D\rightarrow\C$ and $\int_\gamma\Omega_L=\int_\gamma|f||dz|$: writing $\rho=|f|$, it follows that there is no difference between the admissible $\Omega$ used to define extremal volume and the admissible, non-negative, $\rho$ used to define extremal length. 

\paragraph{Example.}All definitions extend to sets $\Lambda$ of rectifiable $n$-currents in $M$. Recall that any homology class $\alpha$ in $M$ can be represented by rectifiable currents (but not necessarily by smooth submanifolds). We will be particularly interested in the case where $\Lambda$ is the set of rectifiable currents in $\alpha\in H_n(M;\Z)$ or in $\alpha\in H_n(M;\R)$. The corresponding extremal volume will be denoted $\mu_\alpha$. 

\ 

From now on we will not distinguish between submanifolds and currents.

\begin{remark}\label{rem:anti_holomorphic}
If we replace $J$ with $-J$, the orientation on $M$ changes by $(-1)^n$. Also, $(n,0)$-forms are swapped with $(0,n)$-forms and $\overline\Omega_M=(-1)^{n^2}\Omega_M$, where the LHS is the $\overline\Omega$-volume form on $(M,-J)$ and the RHS is the $\Omega$-volume form on $(M,J)$. It follows that the $\Omega$-volume of $(M,J)$ coincides with the $\overline\Omega$-volume of $(M,-J)$. The orientation on $L$ is independent of that on $M$, and the $\Omega$-volume of $L$ does not notice the difference between $\Omega$, $\overline\Omega$. We conclude that extremal volume is invariant also under anti-biholomorphisms.

Let $-L$ denote $L$ with the opposite orientation. Then $\Omega_{-L}=-\Omega_L$ but each is positive with respect to its own orientation so the $\Omega$-volume of $L$ is independent of the orientation of $L$. Analogously, given $\Lambda$, let $-\Lambda$ denote the same set of submanifolds, each endowed with the opposite orientation. Then $\mu_\Lambda=\mu_{-\Lambda}$.
\end{remark}

\section{Lower bounds for the extremal volume}\label{s:estimates}

We have mentioned that the notion of $\Omega$-volume, thus of extremal volume,  depends only on the equivalence class of $\Omega$, defined in terms of rotations.

However, the choice of $\Omega$ also allows us to ``organize" $n$-planes in $TM$, defining a ``Grassmannian geometry" specifically sensitive to rotations. Understanding this point will sometimes allow us to obtain lower bounds for the extremal volume. 

Let us denote by $\mathcal{G}$ the Grassmannian of non-oriented $n$-planes in $TM$, and by $\widetilde{\mathcal{G}}$ the Grassmannian of oriented $n$-planes.

\begin{definition} Let $(M^{2n},J)$ be a complex manifold. Fix a smooth section $\Omega$ of $K_M$. An oriented $n$-plane $\pi$ is $\Omega$-\textit{special} if $\Omega_\pi=\Omega$ (restricted to $\pi$). We will denote by $\mathcal{S}_\Omega\subseteq\widetilde{\mathcal{G}}$ the Grassmannian of $\Omega$-special planes in $TM$. 

An oriented submanifold $L^n\subseteq M$ is $\Omega$-\textit{special} if each $T_pL\in\mathcal{S}_\Omega$, \textit{i.e.} $\Omega_L\equiv\Omega$ (restricted to $TL$). 
\end{definition}

A $n$-plane is thus $\Omega$-special (for some orientation) if $\Omega$, on that plane, takes real values, \textit{i.e.} its imaginary part vanishes. Let us look into this more closely.

If $\Omega[p]=0$, any $n$-plane at that point is special with respect to any orientation, \textit{i.e.} $\mathcal{S}_\Omega[p]=\widetilde{\mathcal{G}}[p]$. If a plane at $p$ contains complex lines, it is special with respect to any orientation and any $\Omega$, so it belongs to $\mathcal{S}_\Omega[p]$ for all $\Omega$. 

The special condition is thus of interest mainly in the case when $\Omega[p]\neq 0$ and the oriented $n$-plane is totally real (TR), \textit{i.e.} contains no complex lines. We can then define a \textit{phase} $e^{i\theta}\in \S^1$ such that $\Omega_\pi=e^{i\theta}\Omega$ and $\pi$ is special if and only if $e^{i\theta}=1$: we say it is $\Omega$-\textit{special totally real} (STR). 

Concerning oriented submanifolds, and assuming $\Omega$ never vanishes, we thus notice two interesting situations at opposite extremes of the geometric spectrum.

On the one hand, assume $L$ is complex (thus $n$ is even). It is then special for any $\Omega$ and $\Omega_L\equiv 0$, so $\int_L\Omega_L=0$. In particular, for any $\Omega$, a complex submanifold minimizes the $\Omega_L$-volume when compared to any other oriented $n$-submanifold. The same happens for any $L$ whose tangent bundle contains complex lines at each point.

On the other hand, assume $L$ is totally real. We can then define a \textit{phase function} $e^{i\theta}:L\rightarrow \Sph^1$ such that $\Omega_L=e^{i\theta}\Omega$. $L$ is $\Omega$-STR if the phase function satisfies $e^{i\theta}\equiv 1$.

The latter situation extends the following well-known setup, cf. \cite{HL}. Recall that a \textit{Calabi-Yau manifold} is a complex manifold $(M,J,g,\Omega)$ where $g$ is a K\"ahler Ricci-flat metric and $\Omega$ is a parallel (thus holomorphic) nowhere-vanishing complex volume form. In this case submanifolds which are \textit{special Lagrangian}, \textit{i.e.} simultaneously special and Lagrangian (thus TR) are a classical object of interest because they are ``calibrated'', thus volume-miminizing in their homology class.

We can extend this result to our non-metric context, also allowing non-TR points. We remark that the following result might not seem credible until one notices that any closed $(n,0)$-form is automatically holomorphic: this rigidifies $\Omega$ considerably so that, as a section of $K_M$, it is uniquely defined by its values on any open subset of $M$ or indeed on any open subset of a TR submanifold.

\begin{prop}\label{prop:special_minimizes}
 Let $\Omega$ be closed, equivalently holomorphic. Then any $\Omega$-special submanifold $L$ minimizes the $\Omega$-volume functional within its homology class $\alpha$. In particular, $\Len(\alpha,\Omega)=\int_L\Omega$. 

Setting $\Lambda:=\alpha$, this implies the lower bound
 $$\mu_\alpha\geq(\int_L\Omega)^2/\int_M\Omega_M.$$
\end{prop} 
\begin{proof}
Let $L'$ be any oriented submanifold homologous to $L$. Using the fact that all integrals are real, we obtain
\begin{align*}
 \int_L\Omega_L&=\int_L\Omega=\int_{L'}\Omega=\int_{L'}\Ree(\Omega)+i\int_{L'}\Imm(\Omega)\\
 &=\int_{L'}\Ree(\Omega)\leq\int_{L'}|\Ree(\Omega)|\leq\int_{L'}|\Omega|=\int_{L'}\Omega_{L'}.
\end{align*}
We note that equality holds exactly when $L'$ is also $\Omega$-special.

To compute $\mu_\alpha$ we must consider all admissible forms $\Omega'$ but, since $\mu_\alpha$ is scale invariant, we can normalize them so that $A(\Omega')=A(\Omega)$. The lower bound is then immediate.
\end{proof}

Recall that pluri-potential theory shows that, for any $L$ TR, there exists a small neighbourhood of $L\subseteq M$ which is Stein: in particular, this neighbourhood contains no compact complex $n$-submanifolds. If $L$ is STR then the above proposition proves that the homology class of $L$ contains no such submanifolds. More generally, the same is true for any special submanifold with positive $\Omega_L$-volume.

\begin{remark} 
When $M$ is Calabi-Yau and $L$ is special Lagrangian, $\Omega_{|TL}=\vol_L$, the induced Riemannian volume form on $L$. It follows that $\int_L\Omega_L=\int_L\Omega=\int_L\vol_L$, \textit{i.e.} the $\Omega$-volume of $L$ coincides with the Riemannian volume.
\end{remark}

In order to emphasize the flexibility of our setting, we note that there exists an analogous result for certain classes $\Lambda$ of submanifolds with boundary.
\begin{prop}\label{prop:bdy_special_minimizes}
Let $\Sigma$ be a complex submanifold in $M$. Choose a relative homology class $\alpha\in H_n(M,\Sigma;\Z)$ and let $\Lambda$ denote the class of integral currents in $\alpha$.

Let $\Omega$ be closed, equivalently holomorphic. Then any $\Omega$-special submanifold $L\in\Lambda$ minimizes the $\Omega$-volume functional restricted to $\Lambda$. In particular, $\Len(\Lambda,\Omega)=\int_L\Omega$. This implies the lower bound
 $$\mu_\Lambda\geq(\int_L\Omega)^2/\int_M\Omega_M.$$
\end{prop} 
To prove this result, choose any $L'\in\Lambda$. Then $L-L'=\partial(T^{n+1})+S$, where $S$ is an integral $n$-current contained in $\Sigma$. Since $\Omega$ vanishes on $\Sigma$, the same proof used for Proposition \ref{prop:special_minimizes} applies.

We remark that if $\Lambda$ contains a TR submanifold $L'$, the assumption $\partial L'\subseteq\Sigma$ implies that $\Sigma$ must be a complex hypersurface.

\section{Calculation of the extremal volume}\label{s:extremal}

Proposition \ref{prop:special_minimizes} provides, under appropriate conditions, a cohomological lower bound for extremal volume. This bound depends on the complex structure. 

We now want to find situations in which it is possible to calculate the extremal volume precisely. As mentioned, this is an issue even for extremal length. One context in which the latter can be computed is the case of quadrilaterals, because the Riemann mapping theorem allows us to restrict to the special case of rectangles, which have the property of being fibred by segments parallel to their sides. We generalize this situation as follows.

Assume $M$ has the structure of a fibration over a $n$-dimensional smooth base manifold $B$, with generic fibre $L$. Assuming $L$ is oriented we can orient also $B$, as follows. Choose a nowhere-vanishing volume form $\vol_M$ on $M$, positive with respect to the orientation induced by $J$. Choose any $b\in B$ and let $L_b$ denote the corresponding fibre in $M$.  Let $w_1,\dots,w_n$ be a basis for $T_bB$. Any local trivialization of the fibration allows us to lift the vectors $w_i$ to $M$, obtaining vector fields $\tilde{w}_i$ defined along $L_b$ which project to $w_i$. Now choose any $p\in L_b$ and let $v_1,\dots,v_n$ be a positive basis for $T_pL_b$. We say that the basis $w_i$ is positive if $(-1)^{\frac{n(n-1)}{2}}\vol_M(\tilde{w}_1,\dots,\tilde{w}_n,v_1,\dots,v_n)>0$. One can check that this construction is independent of all choices.

We can build a $n$-form $\Omega_B$ on $B$ through the process of ``vertical integration'' applied to $\Omega_M$: using the above notation,
\begin{equation*}
 \Omega_B[b](w_1,\dots,w_n):=(-1)^{\frac{n(n-1)}{2}}\cdot\int_{L_b}\Omega_M(\tilde{w}_1,\dots,\tilde{w}_n,\cdot,\dots,\cdot)\in\R.
\end{equation*}
Different choices of lifting differ only by vectors in $TL_b$, but integrating over $L_b$ saturates these directions so $\Omega_B$ is well-defined independently of this choice. It is non-negative with respect to the above orientation of $B$, and has the property $\int_B\Omega_B=\int_M\Omega_M$.

\begin{remark}\label{rem:orientations}
The orientation adopted in this section is irrelevant for the other parts of the paper. It is chosen simply to be compatible with the basic example (using compact notation)
\begin{equation*}
 M=\R^{2n}\rightarrow B=\R^n, \ \ (x,y)\mapsto x,
\end{equation*}
where $M$ has the orientation induced from $J$ and $\Omega=(-i)^ndz$ so that each fibre, endowed with the standard orientation, is $\Omega$-STR. In this case $(-1)^{\frac{n(n-1)}{2}}\Omega_M=dx\wedge dy$, $B$ has the standard orientation and $\Omega_B[x](\partial x)=\int_{\R^n}dy$. 

In particular, assume $w_1,\dots,w_n$ is a positive basis for $T_bB$. Set $v_i:=Jw_i$. Then $\Omega(v_1,\dots,v_n)>0$.

Applying a biholomorphism to $M$ may invert the orientation on $L$, thus change the orientation on $B$. For example, if $M=\C^2$ the biholomorphism which exchanges the variables $z_1$ and $z_2$ on $M$ also exchanges the variables $x_1,x_2$ on the base $B$, thus changes its orientation.
\end{remark}

Now assume each fibre $L_b$ is TR. In this case the spaces $J(T_pL_b)$ define a canonical complement of the spaces $T_pL_b$, so we obtain a canonical lift of $w_i$ by prescribing $\tilde{w}_i\in J(TL_b)$. In particular this implies that $L$ is parallelizable: each $T_pL_b$ is canonically isomorphic to $T_bB$ via the map
\begin{equation}\label{eq:canonical_lift}
T_pL_b\rightarrow T_bB,\ \ v\mapsto w:=\pi_*[p](-Jv).
\end{equation}
Given a basis $w_i$ of $T_bB$, we will denote the corresponding vector fields on $L_b$, defined via (\ref{eq:canonical_lift}), by $v_i$. Notice that $-Jv_i$ coincides with the canonical $\tilde{w}_i$ defined above. 

In this context, the process of vertical integration has an extra feature: given that on each fibre $L_b$ we already have the $n$-dimensional form $\Omega_{|TL_b}$, there must exist a ``density function" which relates it to the $n$-form we are integrating along the fibre. Let us compute this function.

\begin{lem}\label{l:Omega_STR}
Let $\Omega$ be a $(n,0)$-form on $M$. Assume $M$ admits a TR fibration. Then, using the above notation, along each fibre $L_b$ and for each basis $w_1,\dots,w_n$ of $T_pB$,
\begin{equation*}
(-1)^{\frac{n(n-1)}{2}}\cdot\Omega_M(\tilde{w}_1,\dots,\tilde{w}_n,\cdot,\dots,\cdot)_{|TL_b}=\bar{\Omega}(v_1,\dots,v_n)\cdot\Omega_{|TL_b}(\cdot,\dots,\cdot).
\end{equation*}
It follows that, for each $b\in B$,
\begin{equation*}
\Omega_B[b](w_1,\dots,w_n)=\int_{L_b}\bar{\Omega}(v_1,\dots,v_n)\cdot\Omega.
\end{equation*}
\end{lem}
\begin{proof} The claim is trivially true wherever $\Omega$ vanishes. 

Assume $\Omega$ does not vanish. We may also assume that $\tilde{w}_i=-Jv_i$ so it suffices to prove 
$$(-1)^{\frac{n(n-1)}{2}}\cdot\Omega_M(-Jv_1,\dots,-Jv_n,a_1,\dots,a_n)=\bar\Omega(v_1,\dots,v_n)\cdot\Omega(a_1,\dots,a_n),$$
for any basis $v_1,\dots,v_n$ of $T_pL_b$ and vectors $a_1,\dots,a_n$ in $T_pL_b$.
We can identify $(T_pM,J)$ with $\C^n$ so that $T_pL_b$ corresponds to $\R^n$, described by variables $y_1,\dots,y_n$, and $\Omega$ coincides with $e^{i\theta}(dz^1\wedge\dots\wedge dz^n)$, for some $\theta$. It follows that $(-1)^{\frac{n(n-1)}{2}}\cdot\Omega_M=dx^1\wedge\dots\wedge dx^n\wedge dy^1\wedge\dots\wedge dy^n$. Let $M\in\GL(n,\R)$ denote the matrix whose columns contain the coordinates of $v_i$ in terms of the basis $\partial y_1,\dots,\partial y_n$. Then $M$ also represents the coordinates of $-Jv_i$ in terms of $\partial x_j$. Let $N\in\gl(n,\R)$ denote the matrix whose columns contain the coordinates of $a_i$ in terms of $\partial y_1,\dots,\partial y_n$. Then the LHS in the above equation is $\det(M)\det(N)$ while the RHS is $(-i)^n\det(M)\,i^n\det(N)$, so they coincide.
\end{proof}

In other words, for each choice of $w_1,\dots,w_n$, $\Omega_B(w_1,\dots,w_n)$ is the integral average of the density function $\bar{\Omega}(v_1,\dots,v_n)$.

\begin{definition} Let $\Omega$ be a $(n,0)$-form on $M$. A TR fibration is $\Omega$-\textit{parallel} if, for each $b\in B$ and using the parallelization defined in (\ref{eq:canonical_lift}),  the density function $\bar{\Omega}(v_1,\dots,v_n):L_b\rightarrow\C$ is constant.
\end{definition}

This condition is independent of the particular basis $w_i$ used to define $v_i$, but does depend on the specific parallelization (\ref{eq:canonical_lift}). In dimension 1 it is analogous to the fact that parallel fibres have constant distance from each other.

Let us check how these various conditions interact. Assume that, for some closed $\Omega$, $M$ admits a parallel STR fibration. If $\Omega$ vanishes at some point, the parallel condition forces it to vanish along the whole fibre. Clearly, on this fibre, $\int_{L_b}\Omega=0$. Since $\Omega$ is closed, the same holds for each fibre so $\Omega\equiv 0$. If we further assume that $\Omega$ is admissible, we reach a contradiction: it follows that $\Omega$ must be nowhere vanishing, so $K_M$ is holomorphically trivial.

We can now show that, in appropriate circumstances, extremal volume is a cohomological quantity (which depends on $J$). More specifically, the lower bound found in Proposition \ref{prop:special_minimizes} is actually an equality.

\begin{thm}\label{thm:calculate}
Assume that, for some closed and admissible $\Omega$, $M$ admits a parallel STR fibration (thus $K_M$ is holomorphically trivial), with generic fibre $L$. Let $\alpha$ be the homology class of the fibres. Then
\begin{equation*}
 \mu_\alpha=\frac{\left(\int_L\Omega\right)^2}{\int_M\Omega_M}.
\end{equation*}
\end{thm}
\begin{proof} $\Omega$ admissible implies that $\int_M\Omega_M>0$. We already know that
$$\mu_\alpha\geq\frac{\left(\int_L\Omega\right)^2}{\int_M\Omega_M}.$$ 

To prove the opposite inequality, choose any admissible $\Omega'$. Up to rescaling we can assume $\Len(\alpha,\Omega')=\int_L\Omega$, so $\int_{L_b}|\Omega'|\geq\int_L\Omega$  for each fibre. Up to rotation we can assume $\Omega'=\rho\,\Omega$, for some non-negative $\rho:M\rightarrow\R$: this allows us to eliminate the norm, obtaining $\int_{L_b}\rho\,\Omega\geq\int_{L_b}\Omega$ thus $\int_{L_b}(\rho-1)\,\Omega\geq 0$. 

Since also $\int_{L_b}(\rho-1)^2\,\Omega\geq 0$, a simple algebraic manipulation now yields, for each fibre, 
\begin{equation*}
\int_{L_b}\rho^2\,\Omega\geq\int_{L_b}\Omega. 
\end{equation*}
Choose a positive basis $w_1,\dots,w_n$ for $T_bB$. As in Remark \ref{rem:orientations}, the corresponding quantity $\Omega(v_1,\dots,v_n)$ is positive and coincides with the constant density function $\bar{\Omega}(v_1,\dots,v_n)$. Multiplying both sides of the above inequality by this quantity and using Lemma \ref{l:Omega_STR}, we find $\Omega_B'\geq \Omega_B$.

Let us now integrate over $B$, obtaining $\int_M\Omega_M'\geq\int_M\Omega_M$. Inverting and multiplying both sides by $(\Len(\alpha,\Omega'))^2$, we find
$$\mu_\alpha\leq\frac{\left(\int_L\Omega\right)^2}{\int_M\Omega_M},$$ 
thus the result.
\end{proof}

Observe that the strategy in this proof is the following. We first argue that we can assume that the fibration is STR with respect to both forms $\Omega$, $\Omega'$. We then want to show that if $\Omega'$ increases the size of each fibre, compared to $\Omega$, then the induced form on $B$ also increases. Setting $f:=\rho^2$ and $g:=\Omega(v_1,\dots,v_n)$, this boils down to the following abstract question: does $\int f\geq \int 1$ imply $\int fg\geq \int g$? (Equivalently, setting $\tilde{f}:=f-1$: does $\int\tilde{f}\geq 0$ imply $\int\tilde{f}g\geq 0$?) In general the answer is no, explaining the importance of the parallel condition.

\begin{remark}\label{rem:parallel}
Using the notation of (\ref{eq:canonical_lift}), the hypotheses of Theorem \ref{thm:calculate} imply that $\Omega_B(w_1,\dots,w_n)=c\cdot\Omega(v_1,\dots,v_n)=c\cdot |\Omega(\tilde{w}_1,\dots,\tilde{w}_n)|$, where $c:=\int_L\Omega$. Normalizing $\Omega$ so that $c=1$, it follows that $B$ is endowed with a canonical volume form such that $\pi$ is a volume-preserving submersion, ie $\pi_*:(J(T_pL_b),|\Omega|)\rightarrow (T_bB,\Omega_B)$ is an isomorphism for all $p\in L_b$. Given an STR fibration, this property provides an alternative characterization of the parallel condition.
\end{remark}

In the situation of Theorem \ref{thm:calculate}, the map $\alpha\rightarrow\mu_\alpha$ has specific properties in terms of the algebraic structure on homology. In particular, $\mu_0=0$ and $\mu_{\lambda\alpha}=\lambda^2\mu_\alpha$. The following is also a simple consequence.

\begin{cor}\label{cor:multiplicative}
Assume manifolds $M_1$, $M_2$ admit parallel STR fibrations with respect to $\Omega_1$, $\Omega_2$ and classes $\alpha_1$, $\alpha_2$, as in Theorem \ref{thm:calculate}. Then the product manifold $M_1\times M_2$ admits a parallel STR fibration with respect to $\pi_1^*\Omega_1\wedge\pi_2^*\Omega_2$ and the class $\alpha_1\times\alpha_2$, where $\pi_i:M_1\times M_2\rightarrow M_i$ denotes the projection maps. 

Furthermore, $\mu_{\alpha_1\times\alpha_2}=\mu_{\alpha_1}\cdot\mu_{\alpha_2}$.
\end{cor}

\section{Examples}\label{s:first_examples}

Theorem \ref{thm:calculate} raises the question of finding examples of manifolds equipped with a parallel STR fibration. The simplest way to ensure the parallel assumption is via an appropriate group action. Specifically, assume $G$ acts holomorphically on $M$, preserves $\Omega$, and its orbits are STR. The induced fibration is then automatically parallel. We illustrate via the following example.

\paragraph{Reinhardt domains.} A \textit{Reinhardt domain} is an open subset of $\C^n$ invariant under the standard action of $(\Sph^1)^n$, \textit{i.e.} the action of matrices $M\in\GL(n,\C)$ of diagonal form $(e^{i\theta_1},\dots,e^{i\theta_n})$.

The simplest case is an annulus $\{r_1<|z|<r_2\}\subseteq\C$. Each circle $|z|=r$ (oriented counter-clockwise) is totally real, and special with respect to the (well-defined, closed, not exact) $(1,0)$-form $\Omega:=-i(d\log z)$. Let $\alpha$ denote the homology class of any such circle. In order to calculate $\mu_\alpha$, it is convenient to take advantage of its invariance under biholomorphisms. We thus change coordinates,
\begin{equation*}
w=t+i\theta\mapsto z:=\exp w
\end{equation*} 
identifying the annulus with $(\log r_1,\log r_2)\times \Sph^1$ so that $\exp^*\Omega=-i dw$. It is now clear that the annulus admits a fibration structure as in Theorem \ref{thm:calculate}. It follows that 
\begin{equation*}
\mu_\alpha=\frac{(-i)^2(2\pi i)^2}{2\pi(\log r_2-\log r_1)}=\frac{2\pi}{\log r_2-\log r_1},
\end{equation*}
showing that extremal volume (in this case, extremal length) can be expressed in terms of ``logarithmic length", \textit{i.e.} length of the base space calculated with respect to the logarithmic variable.

Corollary \ref{cor:multiplicative} shows how to calculate extremal volume for Reinhardt domains given by products of annuli.

A general Reinhardt domain can be viewed as a bundle over a base space in $\R^n$, endowed with variables $(|z_1|,\dots,|z_n|)$, with fibers given by the orbits of the group action. The generic fiber is $(\Sph^1)^n$, but the fiber over a point which has some $|z_i|=0$ collapses to a lower-dimensional torus. We will focus on bounded domains in which no collapsing occurs, \textit{i.e.} the base space is relatively compact in $(\R^+)^n$. Changing coordinates via 
\begin{equation*}
(w_1=t_1+i\theta_1,\dots,w_n=t_n+i\theta_n)\mapsto (z_1:=\exp w_1,\dots, z_n:=\exp w_n)
\end{equation*}
shows that the fibers are parallel STR with respect to the holomorphic $(n,0)$-form $(-i)^n dw^1\wedge\dots\wedge dw^n$. Let $B\subseteq\R^n$, endowed with the variables $(t_1,\dots,t_n)$, denote the new base space and $\alpha$ denote the homology class of the fiber. Then
\begin{equation*}
\mu_\alpha:=\frac{(-i)^{2n}(2\pi i)^{2n}}{(2\pi)^n\int_B\,dt^1\wedge\dots\wedge dt^n}=\frac{(2\pi)^n}{\vol(B)},
\end{equation*}
where again $\vol(B)$ indicates the ``logarithmic volume" of the original base space, appropriately oriented.

The conclusion is that extremal volume provides a geometrically intuitive, easy to calculate, biholomorphic invariant for our class of Reinhardt domains. 

\begin{remark}\label{rem:algebraic_bihols} Consider the biholomorphisms $\phi:M\rightarrow M'$ between Reinhardt domains in $(\C^*)^n$ which preserve the torus action in the following sense: any $g'\in (\Sph^1)^n$ acting on $M'$ can be obtained as $\phi\circ g\circ\phi^{-1}=g'$, for some $g\in(\Sph^1)^n$ acting on $M$, and viceversa. It is shown in  \cite{Shimizu} that such biholomorphisms are generated by (i) dilations of the form $(z_1,\dots,z_n)\mapsto (a_1z_1,\dots,a_nz_n)$, for some $a_i\in\C^*$, (ii) maps of the form 
\begin{equation*}
(z_1,\dots,z_n)\mapsto (z_1^{a_{11}}z_2^{a_{21}}\dots z_n^{a_{n1}},\dots,z_1^{a_{1n}}z_2^{a_{2n}}\dots z_n^{a_{nn}}),
\end{equation*}
where $(a_{ij})\in GL(n,\Z)$ and thus has determinant $\pm 1$. For these specific maps it is simple to check that $\vol(B)$ is preserved, and thus defines an invariant within this specific category. Our calculation shows that extremal volume offers a natural generalization of this invariant to any complex manifold, and that it is invariant under all biholomorphisms.
\end{remark}

In Calabi-Yau geometry and Mirror Symmetry there is strong interest in the existence and properties of special Lagrangian fibrations. Compared to our parallel STR fibrations, this imposes the stronger Lagrangian condition on the fibres but no specific condition on the fibration. The following example shows that these different assumptions, ie Lagrangian fibres and parallel fibration, are independent.

\paragraph{A special Lagrangian fibration.}Consider the fibration \cite{HL}
\begin{equation*}
F:\C^3\rightarrow \R^3, \ \ (z_1,z_2,z_3)\mapsto (|z_1|^2-|z_3|^2, |z_2|^2-|z_3|^2, \Imm(z_1z_2z_3))
\end{equation*} 
Each fibre $L_b:=F^{-1}(b)$, $b\in\R^3$, is special Lagrangian. There is a group $G\simeq \Sph^1\times\Sph^1$ acting holomorphically on $\C^3$ which  preserves $\Omega:=dz^1\wedge dz^2\wedge dz^3$ and the fibres, but the action is not transitive on the fibres. The Lagrangian condition implies that $J(T_pL_b)=(T_pL_b)^\perp$ so it is generated by the gradient vectors $\nabla F^i$. In the notation of (\ref{eq:canonical_lift}) and letting $v_i$ denote the standard basis of $	\R^3$, one can explicitly calculate the appropriate linear combinations $\tilde{w}_i$ of $\nabla F^i$ such that $dF(\tilde{w}_i)=v_i$, then use this to show that this fibration is not parallel.

We remark that any special Lagrangian fibration $\pi:M\rightarrow B$ from a Calabi-Yau manifold onto a Riemannian manifold $B$ such that $\pi$ is also a Riemannian submersion, ie $\pi_*:(T_pL_b)^\perp\rightarrow T_bB$ is an isometry for all $p\in L_b$, would be parallel. This follows from Remark \ref{rem:parallel}.

\ 

Special Lagrangian fibrations are known to be highly constrained both analytically \cite{McLean} and topologically. The following example shows that STR submanifolds and fibrations can instead appear in infinite-dimensional families. Analytically, an important difference between the two cases is that the STR condition is non-elliptic. We refer to \cite{LP1,LP3} for a discussion of this fact in a slightly different context. The example below also emphasizes that the parallel condition is more subtle than might be apparent. 

\paragraph{Elliptic fibre bundles.}Choose a complex torus $\C/\Lambda$, where $\Lambda$ is the lattice generated by $\{1,\tau\}$. Let $E$ be a holomorphic line bundle over the torus and let $E^*$ denote $E$ minus the zero section. Locally, we can identify the total space of $E^*$ with $\C\times \C^*$. Given two local charts with coordinates $(z,w)$ and $(\zeta,\eta)$, the transition functions are of the form $(z,w)\mapsto (z+\lambda,\phi(z)w)$, for some $\lambda\in\Lambda$ and $\phi\in\mathcal{O}^*$. The differential form $\Omega:=dz\wedge (-id\log w)=-i(dz\wedge dw/w)$ on $\C\times\C^*$ is invariant under gluing. It follows that $\Omega$ is a well defined holomorphic $(2,0)$-form on the total space of $E^*$. Choose $c>1$ and consider the action of $\Z$ on $\C^*$ defined by $n\cdot w:=c^n w$. It induces an action on $E^*$ which preserves $\Omega$. Let $M$ denote the quotient complex surface $E^*/\Z$, again endowed with $\Omega$. It is a holomorphic fibre bundle over $\C/\Lambda$ with fibre isomorphic to the  elliptic curve $\C^*/\Z\simeq\C/\Lambda'$, where $\Lambda'$ is the lattice generated by $\{\log c, 2\pi i\}$. The group action of $\C^*/\Z$ endows $M$ with a principal fibre bundle structure.

Assume $E$ is topologically trivial. Choose any smooth section $\sigma$ of $E^*$. Then, for any fixed $b:=(s,t)$ with $s\in [0,1]$ and $t\in [0,\log c]$, $(x,\theta)\mapsto e^{t+i\theta}\cdot\sigma(x+s\tau)$ defines a 2-dimensional torus $L_b$ in $M$. Varying $b$, we obtain a fibration over the torus $B$ defined by the variables $(s,t)$. Notice that if we fix $b,x$, the remaining variable $\theta$ parametrizes a TR circle in the complex fibre $(E^*/\Z)_b$. Varying $x\in [0,1]$ generates a direction transverse to this fibre, so each torus $L_b$ is necessarily TR. Now choose $b\in B$ and $p\in L_b$. In terms of local coordinates $(z,w)\in\C\times\C^*$, $p=(x+s\tau, e^{t+i\theta}\sigma(x+s\tau))$ and $T_pL_b$ is generated by the vectors  $v_1:=(1,e^{t+i\theta}\sigma_x)$, $v_2:=(0,ie^{t+i\theta}\sigma)$, so $\Omega(v_1,v_2)=-idz(v_1)dw(v_2)/(e^{t+i\theta}\sigma)\equiv 1$. It follows that each $L_b$ is STR, so we have obtained infinitely many (depending on $\sigma$) STR torus fibrations of $M$ over the same base space $B$. 

In general however the above vectors are not of the form $v_i=J(\tilde{w}_i)$ required by (\ref{eq:canonical_lift}). Indeed, the projection map onto $(s,t)$ is of the form $(z,w)\mapsto(\pi_\tau(z),\log |w|)$ (up to the quotient operation in the torus  $B$), where $\pi_\tau$ calculates the second coordinate with respect to the basis $\{1,\tau\}$. Linearizing it at $p$ and applying it to $-iv_1,-iv_2$ yields vectors in $T_bB$ which are not independent of $x,\theta$, ie of $p\in L_b$. It follows that the fact $\Omega(v_1,v_2)\equiv 1$ does not prove that these fibrations are parallel, and indeed they are not.

On the other hand one can show (\cite{BHPV}, p. 197) that, whenever $E$ is topologically trivial, $M$ is a complex torus. We will show in Section \ref{s:example}, via a different construction, that any such torus also admits parallel STR fibrations.

If $E$ is not topologically trivial, the above construction produces examples of ``primary Kodaira surfaces". Any such surface is non-K\"ahler. In this case it is not clear if one can produce a global STR fibration. However, recall that any $\Sph^1$-bundle over $\Sph^1$ is trivial. It follows that we can produce a section $\sigma$ of the unitary subbundle in $E^*$ (for some metric), restricted to the circle in $\C/\Lambda$ parametrized by $x\mapsto (x+s\tau)$ (for any fixed $s$). We thus obtain a $\Omega$-STR torus as above, which we can use to obtain lower bounds for the corresponding extremal volume $\mu_\alpha$ as in Section \ref{s:estimates}.

\section{Moduli spaces of complex tori}\label{s:example}

Another class of examples endowed with a group action is provided by complex tori $\C^n/\Lambda$. Here, $\Lambda$ is a lattice in $\C^n=\R^{2n}$ of maximal rank, \textit{i.e.} an additive subgroup generated by $2n$ $\R$-linearly independent vectors. 

Compared to Reinhardt domains, this class has two important features: (i) complex tori are Lie groups, so the group action here is much stronger, (ii) the manifold is compact so, up to normalization, any holomorphic volume form is a (constant) rotation of $dz^1\wedge\dots\wedge dz^n$. It will thus suffice to focus on this form. Our interest in this class stems from the fact that it will suggest a new twist to the theory of extremal volume.

For $n=1$, each integral homology class $\alpha\in H_1(\C/\Lambda;\Z)$ can be represented by a segment in $\C$ which connects $0$ to an element of the lattice, so we can identify $H_1(\C/\Lambda;\Z)\simeq\Lambda$. Choose any such $\alpha\simeq\lambda=|\lambda| e^{i\theta}$. The corresponding closed curve in $\C/\Lambda$ is $\Omega$-STR, where $\Omega=e^{-i\theta}dz$ is closed. By translation we obtain a parallel STR fibration of $\C/\Lambda$.

More generally, any integral homology class $\alpha\in H_n(\C^n/\Lambda;\Z)$ can be represented by a subtorus generated by $n$ $\R$-linearly independent vectors in the lattice. The subtorus is TR if the vectors are $\C$-linearly independent. As above, it is then automatically STR for an appropriate choice of closed (constant) $\Omega$. Notice that, in order to properly define the base space $B$, one should keep in mind that the fibres may wrap multiple times around the torus.

The corresponding $\mu_\alpha$ can be calculated using Theorem \ref{thm:calculate} and rotation-invariance of the relevant quantities. We summarize as follows. 

\begin{cor}\label{cor:tori_extremalvol}
The extremal volume of any class $\alpha$ represented by a TR subtorus $L$ in $\C^n/\Lambda$ can be explicitly calculated in terms of the standard volume form $\Omega=dz^1\wedge\dots\wedge dz^n$. Specifically,
 \begin{equation*}
  \mu_\alpha=\frac{(\int_L|\Omega|)^2}{\int_M\Omega_M}.
 \end{equation*}
For all other classes, $\mu_\alpha=0$.
\end{cor}

\begin{remark} Recall that generic complex tori do not admit complex submanifolds. On the other hand, any $\Lambda$ admits a complex basis, thus a TR subtorus. It follows that some $\mu_\alpha$ is always non-trivial. This highlights the usefulness, in the context of non-algebraic manifolds, of an invariant based on TR, rather than complex, submanifolds. 
\end{remark}

Since $\mu_\alpha$ is biholomorphically invariant, we can make use of the known classification results for complex tori in order to simplify the presentation of the torus. Studying these moduli spaces of complex tori raises also a new question, which is foundational for Section \ref{s:cpx_systole}.

\ 

\noindent\textbf{Question:} How does extremal volume behave with respect to variations of $J$? 

\ 

The 1-dimensional case, explained below, is classical: a certain upper bound is uniform with respect to the complex structure. Lemma \ref{l:polartori} will provide an analogue in higher dimensions.

\paragraph{Dimension 1.} Recall that any holomorphic map $f:\C/\Lambda_1\rightarrow\C/\Lambda_2$ lifts to an affine holomorphic function $\tilde{f}(z)=z_0 z+z_1:\C\rightarrow \C$. Up to translations on $\C/\Lambda_2$ we may assume $\tilde{f}(z)=z_0 z$. The map $f$ is a biholomorphism if and only if $\tilde{f}(\Lambda_1)=\Lambda_2$, \textit{i.e.} $\Lambda_2=z_0\cdot\Lambda_1$ for some $z_0\in\C^*$. For example, assume $\Lambda_2$ is obtained via reflection of $\Lambda_1$ across $\R$. The corresponding tori are then generally only complex-conjugate, not biholomorphic, unless $\Lambda_1$ is invariant under this reflection.

The moduli space $\mathcal{M}$ of complex tori can now be built as follows. Any lattice $\Lambda$ in $\C=\R^2$ can be identified with an orbit of the right action of $GL(2,\Z)$ on the Stiefel space $L(\R^2)\simeq GL(2,\R)$ of linear bases on $\R^2$. As seen above, biholomorphic tori are obtained via complex multiplication, which induces the left action of $\C^*$ on $L(\R^2)$. The moduli space of complex tori is thus the double quotient
\begin{equation*}
\mathcal{M}=GL(1,\C)\backslash \GL(2,\R)/GL(2,\Z).
\end{equation*}
This space becomes more concrete if we choose a ``canonical" $\Lambda$ in each biholomorphic equivalence class.  It is well known that such $\Lambda$ can be defined via bases of the form $\{1,\tau\}$, where $\tau$ belongs to the domain $\mathcal{D}\subseteq\C$ defined by the conditions $\Imm\tau>0$, $|\Ree \tau|\leq 1/2$ and $|\tau|\geq 1$. In particular, any $\tau\in\mathcal{D}$ has the property $\Imm(\tau)=|\tau|\sin\theta\geq \sqrt{3}/2$. Up to appropriate identifications along the boundary, this domain exactly parametrizes $\mathcal{M}$.

For example, consider the torus corresponding to $\tau=x+iy\in\mathcal{D}$, \textit{i.e.} generated by $\{1,\tau\}$. As seen above, the conjugate torus is generated by $\{1,x-iy\}$. The basis $\{1, -x+iy\}$ generates the same lattice: we thus see that this torus corresponds to $\tau'=-x+iy\in\mathcal{D}$. In summary, all such pairs $\tau$, $\tau'$ correspond to complex conjugate tori (not biholomorphic unless $x=0$).

Given any complex torus, we can now express any of its extremal lengths in terms of $\tau$. For example,  let $\alpha\in H_1$ be the class of the segment ending in $1$. The corresponding subtorus is STR with respect to $\Omega=dz$. Let us apply Theorem \ref{thm:calculate} (or Corollary \ref{cor:tori_extremalvol}), noticing that $\int_M\Omega_M$ is simply the Euclidean area of the fundamental domain generated by $1,\tau$. We thus obtain, for all complex tori, the uniform bound
\begin{equation}\label{eq:mu1bounded}
\mu_{\alpha}=\frac{1}{|\tau|\sin\theta}\leq \frac{2}{\sqrt{3}}.
\end{equation}
Analogously,  let $\alpha'\in H_1$ be generated by the segment ending in $\tau$. The corresponding subtorus is STR with respect to $\Omega'=e^{-i\theta} dz$. As above, we find $\mu_{\alpha'}=|\tau|/\sin\theta$.
 
Notice: the complex parameter $\tau$ and the real parameters arising from extremal length provide different invariants for complex tori. One can ask the following question: does extremal length provide a complete set of invariants? With the above two choices of reference classes, we have reconstructed $\tau$ up to the ambiguity between $\theta$ and $\pi-\theta$ so the  answer is yes, up to complex conjugation. According to Remark \ref{rem:anti_holomorphic}, this is the best we can hope for.

\paragraph{Higher dimensions.} The fundamental facts concerning complex tori $\C^n/\Lambda$ are similar. Any biholomorphism lifts to a complex affine transformation of $\C^n$, and the full moduli space can be identified with the double quotient of the Stiefel space $GL(n,\C)\backslash \GL(2n,\R)/GL(2n,\Z)$. The natural topology of this space is however not Hausdorff, so one generally prefers to restrict to special subclasses.

Let us consider the class of \textit{principally polarized abelian varieties}, \textit{i.e.} complex tori $(\C^n/\Lambda,\omega)$ endowed with a symplectic structure $\omega$ with the following property: there exists some basis of $\Lambda$ of the form $\{v_1,\dots,v_n,Jv_1,\dots,Jv_n\}$ with respect to which the matrix of $\omega$ has the standard form 
\begin{equation*}
\omega\simeq\left(\begin{array}{cc}
O&I\\
-I&O
\end{array}\right). 
\end{equation*}
In particular $\omega$ is integral on $\Lambda$, so these tori are projective. Two such tori are equivalent if they are related by a biholomorphism which also preserves the symplectic structures. They are parametrized, up to equivalence, by lattices of the form $\Z^n+\tau\cdot\Z^n$, where $\tau=A+iB\in GL(n,\C)$ is a complex matrix in the \textit{Siegel domain} defined by the conditions $A^t=A$, $B^t=B$, $B$ positive definite. Two such lattices, corresponding to $\tau$ and $\tau'$, define the same polarized tori if $\tau$ and $\tau'$ are in the same orbit of a certain action of $Sp(2n,\Z)$ on the Siegel domain.

A fundamental domain $\mathcal{D}$ for this action is known. We are interested in the fact, cf. \cite{klingen} Section 1.3  Lemma 2, that any $\tau\in\mathcal{D}$ satisfies a uniform lower bound $\det(B)\geq c(n)$. Choose $\alpha\in H^n(\C^n/\Lambda;\Z)$ represented by the subtorus defined by $\Z^n$. It is STR with respect to $\Omega=dz^1\wedge\dots \wedge dz^n$, and $\int_M\Omega_M=\det\left(\begin{array}{cc}I&A\\0&B\end{array}\right)$. Theorem \ref{thm:calculate} (or Corollary \ref{cor:tori_extremalvol}) now yields a uniform upper bound of the form 
$$\mu_\alpha=\frac{1}{\int_M\Omega_M}=\frac{1}{\det(B)}\leq d(n).$$ 

\begin{remark}\label{rem:lagrangian}
The subtorus defined by $\Z^n$ is actually special Lagrangian with respect to the polarization on $\C^n/\Lambda$.
\end{remark}

For our purposes, the polarization serves only to help define a useful moduli space, so the above classification is finer than necessary. Forgetting the symplectic structure, we will be content with the following summary.
\begin{lem}\label{l:polartori} There exists $d=d(n)$ such that any complex torus $\C^n/\Lambda$ which admits a principal polarization contains a TR torus whose homology class $\alpha\in H_n(\C^n/\Lambda;\Z)$ has positive extremal volume and satisfies
\begin{equation}\label{eq:mubounded}
\mu_\alpha=\frac{1}{\int_M\Omega_M}=\frac{1}{\det(B)}\leq d.
\end{equation}
\end{lem}
Notice that any non-TR subtorus will yield an extremal volume with value zero. Forgoing this trivial case, the lemma shows that the set of positive values attained by the extremal volume cannot uniformly ``float off'' to infinity, as $J$ varies in this moduli space.

\section{Complex systolic inequalities}\label{s:cpx_systole}

Up to here we have presented extremal volume as a tool for extracting information from a given complex manifold. The previous section suggests however that it can also be used to encode information regarding moduli spaces. Specifically, (\ref{eq:mu1bounded}) and (\ref{eq:mubounded}) show that certain bounds on the geometry of the fundamental domain can be expressed in terms of extremal volume. In the 1-dimensional case there exists a classical alternative way of formulating bounds such as (\ref{eq:mu1bounded}) using the language of ``systolic geometry". Reviewing this concept will suggest a notion of ``complex systolic geometry'', thus providing an analogous reformulation of  (\ref{eq:mubounded}). The definition will require only the lower bounds provided in Section \ref{s:estimates}. 

\paragraph{Riemannian systolic geometry.}Let $M$ be a compact $m$-dimensional differentiable manifold. Given a Riemannian metric $g$ on $M$, one defines the \textit{systole} $s(g)$ of $(M,g)$ to be the smallest length of all non-contractible curves in $M$. One can show that such a minimizing curve exists; it is necessarily a geodesic. In higher dimensions this notion is usually generalized via homology classes: the \textit{$k$-systole} $s_k(g)$ is the smallest $k$-volume amoung all $k$-dimensional cycles representing any non-zero class in $H_k(M;\Z)$. Roughly speaking, the restriction to non-zero classes eliminates values arising from submanifolds generated by possible ``bubbles" in the metric, but having no topological significance. 

The main point of systolic geometry is to relate the values of these $k$-systoles, for various $k$ and $g$. The simplest case concerns the relation between $s(g)$ and $s_m(g)$, \textit{i.e.} the volume of $(M,g)$. Let us consider the case $m=2$. 
 
In two dimensions, given a compact surface $M$, systolic geometry establishes inequalities of the form 
\begin{equation}\label{eq:systolicineq}
\sup_g\frac{s(g)^2}{\int_M\vol_g}\leq c, 
\end{equation}
for some constant $c=c(M)$, as $g$ varies amoung all metrics on $M$. 

An upper bound of this type should not be taken for granted. Roughly speaking, it says that the area of $M$ is uniformly controlled by the length of its shortest non-contractible geodesic, in the sense that $\int_M\vol_g\geq d\cdot s(g)^2$, where $d$ is independent of $g$. In particular, on a topological level this basically implies that the manifold is in some sense ``generated" by non-contractible curves, otherwise these curves could not hope to control the total area. This is clearly false for $M=\Sph^2$, which has no non-contractible curves. 

We remark that one can also think of (\ref{eq:systolicineq}) as a boundaryless analogue of the classical isoperimetric problem, but notice that inequalities of the form (\ref{eq:systolicineq}) are opposite those which appear in isoperimetric problems. 

In $m$ dimensions the analogous inequality would be $s(g)^m/\int_M\vol_g\leq c$. Again, examples of the form $M=\Sph^1\times\Sph^2$ show that such upper bounds, uniform with respect to $g$, are in general impossible unless $M$ has special topological properties. From this point of view, the aim of systolic geometry is to use Riemannian metrics to obtain information on the topology of a given differentiable manifold. The lack of such bounds corresponds to the notion of ``systolic freedom", and reveals interesting connections between systolic inequalities and algebraic topology. We refer to \cite{berger} for a gentle introduction to systolic geometry and for further references.

\paragraph{Loewner's theorem.}The quantities appearing in (\ref{eq:extremallength}) and  (\ref{eq:systolicineq}) are very similar, though the former is more restrictive regarding the class of metrics. This implies that extremal length can provide a useful tool for systolic geometry. 

The case $M=\Sph^1\times\Sph^1$ is a well known example. In 1949, Loewner showed that such $M$ satisfies the systolic inequality (\ref{eq:systolicineq}). This can be proved as follows.

Any Riemannian metric $g$ on $M$ defines a notion of $\pi/2$-rotation, so it defines a complex structure $J$ which is integrable for dimensional reasons. The Riemann-Roch theorem implies that $(M,J)$ is biholomorphic to a torus of the form $\C/\Lambda$, for some $\Lambda$, and $g$ corresponds to $\rho^2g_{std}$, for some $\rho$. As in Section \ref{s:example}, we may assume that $\Lambda$ is generated by $\{1,\tau\}$. Using $\alpha$ as in (\ref{eq:mu1bounded}), by definition $s(g)\leq\Len(\alpha,\rho\,dz)$ so
\begin{equation*}\label{eq:loewner}
\frac{s(g)^2}{\int_M\vol_g}\leq \frac{(\Len(\alpha_1,\rho\,dz))^2}{\int_M\rho^2dxdy}\leq\mu_{\alpha}\leq \frac{2}{\sqrt{3}}.
\end{equation*}
The systolic inequality (\ref{eq:systolicineq}) follows, with $c=2/\sqrt{3}$.

In other words, Loewner's result is the Riemannian version of a holomorphic fact concerning the moduli space of tori. The Riemannian formulation (\ref{eq:systolicineq}) contains basically the same information as the holomorphic formulation (\ref{eq:mu1bounded}). 

\begin{remark}\label{rem:systolic_supremacy}
As is usual for extremal length, (\ref{eq:mu1bounded}) is formulated in terms of a specific homology class $\alpha$, defined in terms of a ``canonical", but still somewhat arbitrary, choice of fundamental domain for the moduli space. The systolic inequality (\ref{eq:systolicineq}) provides a more intrinsic, coordinate-free, formulation of a geometric property of the moduli space.  The class $\alpha$ could alternatively be defined as the minimizer of the extremal length functional on (non-zero) homology, bridging this difference.
\end{remark}

\paragraph{Complex systolic geometry.}We have already mentioned that, in higher dimensions, our notion of extremal volume has no Riemannian content. A direct analogy with the 1-dimensional situation and Loewner's theorem is thus not possible. In the spirit of Remark \ref{rem:systolic_supremacy}, however, the following notion of complex systolic geometry offers a canonical, geometric, reformulation of bounds such as (\ref{eq:mubounded}), applicable to any complex manifold admitting holomorphic volume forms.

\begin{lem}\label{lemma:constantphase}
Assume $\Omega$ is closed and not identically zero. Let $\Omega':=e^{i\theta}\Omega$, for some function $e^{i\theta}:M\rightarrow\Sph^1$. Then $\Omega'$ is closed if and only if $\theta$ is constant.
\end{lem}
\begin{proof} Since $\Omega$ is closed and of type $(n,0)$, we find
\begin{equation*}
d\Omega'=ie^{i\theta}d\theta\wedge\Omega=ie^{i\theta}\bar\partial\theta\wedge\Omega.
\end{equation*}
The right hand side vanishes if and only if either $\Omega=0$ or $\bar\partial\theta=0$. $\Omega$ can only vanish in isolated points. At all other points $\theta$ is a real-valued holomorphic function, thus it is constant. By continuity, it is constant on $M$.  
\end{proof}

\begin{definition}\label{def:complexsystole} 
Let $M$ be a differentiable manifold. Let $J$ be an integrable complex structure on $M$ and $\Omega$ be a closed (thus holomorphic) $(n,0)$-form, not identically zero.

Let $H_n(\Omega)\subseteq H_n(M;\Z)$ denote the set of all homology classes $\alpha$ containing a $e^{i\theta}\Omega$-STR representative $L$, for some constant $\theta$.

If $H_n(\Omega)$ is not empty, we define the $(J,\Omega)$-systole $s(J,\Omega)$ of $M$ to be the smallest value of $\int_L\Omega_L$ amoung all $L$ representing any $\alpha\in H_n(\Omega)$. Equivalently (using the fact that extremal volume is rotation-invariant),
\begin{equation*}
 s(J,\Omega):=\inf\{\Len(\alpha,\Omega):\alpha\in H_n(\Omega)\}.
\end{equation*}

If $H_n(\Omega)$ is empty, we set $s(J,\Omega):=\infty$.

Let $\mathcal{M}$ be a fixed moduli space of integrable complex structures on $M$. We define the extremal complex systole of $\mathcal{M}$ as
\begin{equation*}
  \sigma_\mathcal{M}:=\sup_{(J,\Omega)}\frac{s(J,\Omega)^2}{\int_M\Omega_M},
 \end{equation*}
where $J\in\mathcal{M}$ and $\Omega$ is a closed $(n,0)$-form, not identically zero.
 \end{definition}

The quantities $s(J,\Omega)$, $\sigma_\mathcal{M}$ are thus the ``relaxed" forms of  $\Len(\alpha,\Omega)$, $\mu_\alpha$, with fewer restrictions on $\alpha$, $J$.

The logic underlying this definition is as follows. (i) In the proof of Loewner's theorem, $g$ corresponds to a pair $(J,\Omega)$. In this sense, $s(J,\Omega)$ is the higher-dimensional complex analogue of the Riemannian systole $s(g)$ and the inequality $\sigma_\mathcal{M}\leq c$ is the complex analogue of the Riemannian systolic inequality (\ref{eq:systolicineq}). (ii) Lemma \ref{lem:systole_positive} shows that the restriction to closed $\Omega$ ensures that $s(J,\Omega)$ is positive (rather than non-negative). (iii) Including different phases implies that the complex systole is rotation-invariant, as is true for the extremal volume. This allows comparisons between these two quantities. Lemma \ref{lemma:constantphase} guarantees that it suffices to concentrate on constant phases.

\begin{remark}
It follows from Remark \ref{rem:anti_holomorphic} that if $L$ is $\Omega$-special, then $-L$ is $(-\Omega)$-special.

Notice also that submanifolds which are STR with respect to different phases cannot belong to the same homology class. Indeed, assume there exist an $\Omega$-special submanifold $L$ and an $(e^{i\theta}\Omega)$-special submanifold $L'$ in the same homology class $\alpha$. Then $\int_L\Omega_L=\int_L\Omega=\int_{L'}\Omega=e^{-i\theta}\int_{L'}e^{i\theta}\Omega=e^{-i\theta}\int_{L'}\Omega_{L'}$. Since both $\int_L\Omega_L$ and $\int_{L'}\Omega_{L'}$ are real and positive, it follows that $e^{i\theta}=1$. 
\end{remark}

\begin{lem}\label{lem:systole_positive}
For any $(J,\Omega)$ as in Definition \ref{def:complexsystole}, the complex systole $s(J,\Omega)$ is strictly positive.
\end{lem}

\begin{proof} Since $H_n(M;\Z)$ is a finitely generated Abelian group, the subgroup generated by $H_n(\Omega)$ is also finitely generated. Let $\alpha_1,\dots,\alpha_m$ be generators of this subgroup. Since each $\alpha_j$ is a finite linear combination of elements in $H_n(\Omega)$, up to increasing the number of generators we may assume that each $\alpha_j\in H_n(\Omega)$. It follows that each $\alpha_j$ is represented by a $e^{i\theta_j}\Omega$-STR submanifold $L_j$. Rotation-invariance and Proposition \ref{prop:special_minimizes} show that $l(\alpha_j,\Omega)=l(\alpha_j,e^{i\theta_j}\Omega)=\int_{L_j}e^{i\theta_j}\Omega$, which is positive because $\Omega$ is holomorphic, thus has only isolated zeroes. The systole $s(J,\Omega)$ thus coincides with the minimum value of a finite set of positive numbers.
\end{proof}

\begin{remark} A variation of Definition \ref{def:complexsystole} might include in $H_n(\Omega)$ all classes $\alpha$ represented by a $e^{i\theta}\Omega$-special submanifold $L$, for some constant $\theta$, adding the additional condition $\int_L\Omega_L>0$. This variation includes submanifolds which are not necessarily TR but excludes complex submanifolds (which are trivially special). The analogue of Lemma \ref{lem:systole_positive} would again lead to a strictly positive systole.
\end{remark}

We can now reformulate Lemma \ref{l:polartori} as a complex systolic inequality, freeing it from any particular homology class or coordinate system.

\begin{thm}\label{thm:tori_systolicineq}
 Let $M=\Sph^1\times\dots\times\Sph^1$ be the smooth $2n$-dimensional torus. Let $\mathcal{M}$ be the moduli space of principally polarized complex structures on $M$. There exists $d=d(n)$ such that
 \begin{equation*}
  \sigma_\mathcal{M}\leq d.
 \end{equation*}
\end{thm}
\begin{proof}
 Given any $J$ such that $(M,J)=\C^n/\Lambda$ admits a principal polarization, we may choose $\Omega=dz^1\wedge\dots\wedge dz^n$. Choose $\alpha$ as in Lemma \ref{l:polartori}, represented by an $\Omega$-STR subtorus $L$. By definition, $s(J,\Omega)\leq \int_L\Omega_L=\int_L\Omega$ so 
\begin{equation*}
 \frac{s(J,\Omega)^2}{\int_M\Omega_M}\leq \frac{(\int_L\Omega)^2}{\int_M\Omega_M}=\mu_\alpha\leq d,
 \end{equation*}
where $d$ is as in Lemma \ref{l:polartori}. Since this holds for all $J$, we obtain the desired result. 
\end{proof}

Roughly speaking this shows that, for all $J\in\mathcal{M}$ and all holomorphic volume forms, the complex volume of $M$ is uniformly controlled by the complex volume of its smallest STR submanifold. 

\begin{remark}
A notion of complex systole in the restricted, strongly Riemannian, context of Calabi-Yau manifolds appears in \cite{fan}, cf. Definition 1.1, where the author minimizes an analogue of the $\Omega$-volume only amoung special Lagrangian submanifolds. This definition is then used in \cite{haiden}, whose Theorem 4.2 is similar to our Theorem \ref{thm:tori_systolicineq} but relies on the special fact that any complex torus is a Calabi-Yau manifold. The main focus of these papers is orthogonal to ours: they study symplectic geometry, Fukaya categories and Bridgeland stability conditions.
\end{remark}

\bibliographystyle{amsplain}
\bibliography{capacity}

\end{document}